\documentclass[12pt,a4paper]{article}
\usepackage{fancyhdr}
\usepackage[bottom=1.5in, top=1.25in]{geometry}
\usepackage{fancyhdr}

\usepackage{amssymb,amsmath,amsthm}
\usepackage[english]{babel}
\usepackage{t1enc}
\usepackage[utf8]{inputenc}
\usepackage{epsfig}
\usepackage{subfig}
\usepackage{epstopdf}
\usepackage{xcolor}
\usepackage{comment}
\usepackage{hyperref}
\usepackage[shortlabels]{enumitem}
\usepackage{titlesec}
\titleformat{\subsubsection}[runin]{}{}{}{}[]

\newtheorem{thm}{Theorem}
\newtheorem{cor}[thm]{Corollary}
\newtheorem{lem}[thm]{Lemma}
\newtheorem{prop}[thm]{Proposition}

\theoremstyle{definition}
\newtheorem{defi}[thm]{Definition}
\newtheorem{rem}[thm]{Remark}
\newtheorem{ex}[thm]{Example}

\usepackage{indentfirst}
\def\red1{\color{red}1\color{black}}
\def\blue1{\color{blue}1\color{black}}

\def\01{0\text{-}1}

\makeatletter
\def\blfootnote{\gdef\@thefnmark{}\@footnotetext}
\makeatother

\title{Flip dynamics on perfect matchings beyond bipartite and planar graphs}
\author{Vladimir Bo\v skovi\'c \thanks{Universit\'e Paris-Saclay, CNRS, CEA, Institut de Physique Th\'eorique, 91191 Gif-sur-Yvette, France\\Email address: vladimir.boskovic at ipht dot fr}}

\begin{document}

\maketitle

\begin{abstract}
     We study the flip dynamics on perfect matchings of graphs, where a flip consists of replacing the edges of a perfect matching along an even cycle with the complementary alternating edges. In particular, we want to bound the minimum length of cycles such that any two perfect matchings are related by flips of such cycles.
     
     Given a finite graph $G$, we consider the families of decorated graphs obtained by replacing each vertex of $G$ with a decoration satisfying suitable conditions on the existence of perfect matchings in its subgraphs. We prove a general upper bound for this family. We then obtain stronger bounds for two families of decorations: clique decorations and Fisher decorations, the latter under the assumption that the underlying graph is planar. In both cases, the bounds are independent of the sizes of the decorations.
\end{abstract}

\section{Introduction}

Let $G$ be a finite graph. A \emph{perfect matching} in a graph is a subset of edges such that each vertex is incident to exactly one edge. Let $\mathcal{M}(G)$ denote the space of perfect matchings of $G$, and let $M(G) = |\mathcal{M}(G)|$ be the number of perfect matchings of $G$. 

We consider the \emph{flip dynamics} on the space of perfect matchings $\mathcal{M}(G)$. A flip consists of selecting an even cycle $C$ in $G$ such that every second edge of $C$ belongs to the current perfect matching and then replacing these edges with the complementary set of edges in $C$. We aim to minimize the length of cycles required so that such flips suffice to connect any two perfect matchings of $G$.

This dynamics is of particular interest because it induces a Markov chain on $\mathcal{M}(G)$. The irreducibility of this chain is equivalent to the property that any two perfect matchings can be connected by a sequence of flips. The mixing times for these Markov chains have been studied for different families of graphs, in particular for lozenge tilings~\cite{LaslierToninelli23,AggarwalToninelli26}. In the probabilistic literature the term \emph{Glauber dynamics} is commonly used. Here, we use the term flip dynamics to emphasize the combinatorial nature of our results.

In the planar bipartite setting, the height function introduced by Thurston~\cite{Thurston90} implies that the space of perfect matchings is always connected by flips on the inner faces of the graph. This problem in the non-bipartite or non-planar setting has remained open. 

For finite portions of the triangular lattice, Kenyon and Rémila~\cite{KenRemila} proved that the space of perfect matchings is flip-connected using flips along cycles of length at most $8$. More recently, R{\o}ising and Zhang~\cite{RoisingZhang} studied several classes of non-bipartite planar lattices, such as the Kagome lattice and the truncated hexagonal lattice, and described the cycles that ensure flip-connectedness. 

Hartarsky, Lichev, and Toninelli~\cite{HarLichTon2024} studied this question for hypercubes. They also strengthened the result for the triangular lattice by showing that for parallelogram-shaped portions, flips on cycles of length at most $6$ suffice. Among other examples, Petrov~\cite{Petrov_simulations} has implemented the flip dynamics for arbitrary portions of the triangular lattice.

To state our results, we introduce the notion of a \emph{cycle basis}. It is defined as a minimal set of cycles of $G$ that expresses every even subgraph (a subgraph of $G$ with all its vertices of even degree) as a symmetric difference over the cycles in the basis. A canonical way to define the basis is by taking a spanning tree of $G$ and adding any remaining edge that uniquely defines a cycle. If $G$ is a planar graph, then the most natural choice for its cycle basis is the set of its inner faces.

We study the flip dynamics on two families of decorated graphs. The first family consists of clique-decorated graphs denoted by $G_{clique}$, in which every vertex of degree $d$ is replaced with a complete graph $K_d$. See Figure \ref{fig:K4deco} for an example. Under the following assumptions, we can show that the size of the cycles does not depend on the degrees of the vertices in $G$.

\begin{figure}[ht]
    \centering
    \includegraphics[width=0.85\textwidth]{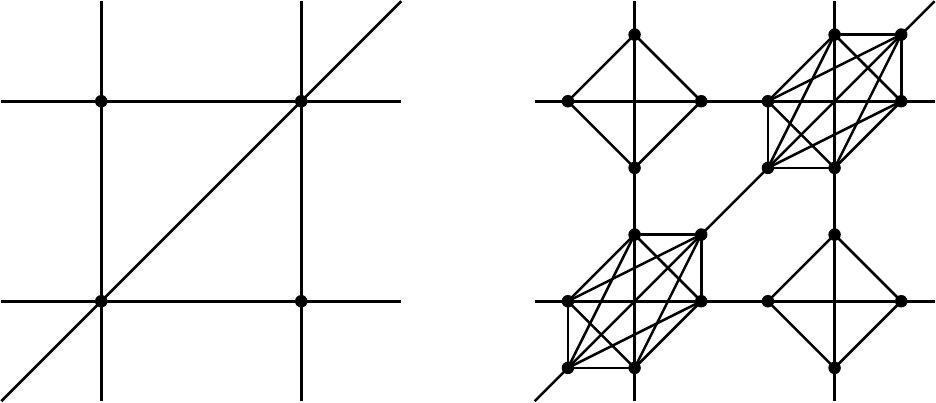}
    \caption{\label{fig:K4deco} An example of a portion of a clique decorated graph.}
\end{figure}

\begin{thm} \label{mainthm2}
    Let $G$ be a graph in which all vertex degrees have the same parity. Then  $\mathcal{M}(G_{clique})$ is flip-connected on the cycles of length at most \[\max_{1 \leq i \leq k}3|C_i|,\] where $\{C_1, \dots, C_k\}$ is a cycle basis of $G$. 
\end{thm}

The second family consists of \emph{Fisher graphs} denoted by $F_G$, introduced by Fisher~\cite{Fisher66} to describe the correspondence between the Ising model of $G$ and the dimer model of $F_G$. Let $G$ be a graph equipped with a fixed cyclic ordering of the edges incident to each vertex, induced by a drawing of $G$ in the plane (or embedding, if $G$ is planar). The Fisher graph $F_G$ is constructed as follows. 

For each vertex $u \in V(G)$ of degree $d$, let $e_1, \dots, e_d$ be the edges incident to $u$, listed according to their cyclic order around $u$. We replace $u$ with a cycle of length $2d$ on the vertex set $\{u_1, u_2, \dots, u_{2d}\}$, whose vertices are arranged in cyclic order consistent with the order of the edges $e_1, \dots, e_d$. For each $i \in \{1, \dots, d\}$, we replace $e_i$ with a long edge incident to $u_{2i-1}$, so that every vertex $u_{2i-1}$ is incident to exactly one long edge. Moreover, for all $1 \leq i \leq d$, we add an edge between $u_{2i}$ and $u_{2i+2}$ (indices modulo $2d$). We refer to the graph that replaces the vertex $u$ as the \emph{Fisher decoration}. See Figure \ref{fig:fishdeg5} for an example when $d = 5$.

\begin{figure}[ht]
    \centering
    \includegraphics[width=0.65\textwidth]{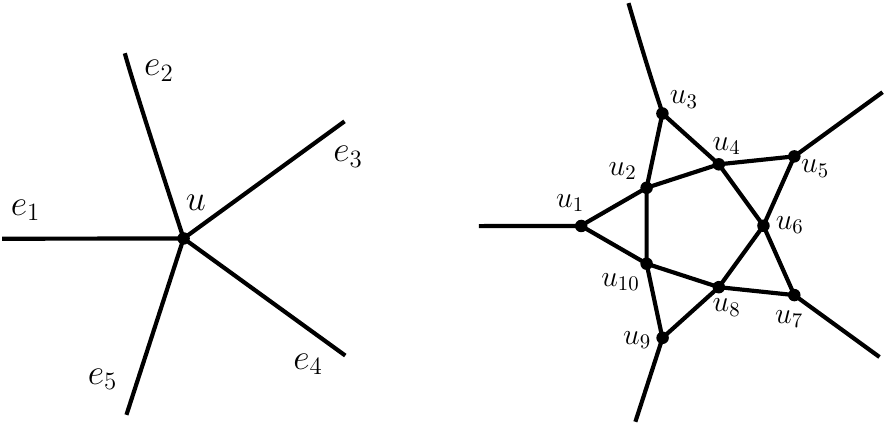}
    \caption{\label{fig:fishdeg5} The Fisher decoration on the vertex of degree $5$.}
\end{figure}
    
\begin{thm} \label{mainthm3}
    Let $G$ be a planar graph on $n$ vertices. Then $\mathcal{M}(F_{G})$ is flip-connected on the cycles of length at most
    \[\max(\max_{1 \leq i \leq k} 4|F_i|,\max_{1 \leq j \leq n} 2\deg(v_j)),\]
    where $\{F_1, \dots, F_k\}$ denotes the set of inner faces of $G$. 
\end{thm}

The rest of the paper is organized as follows.
    
In Section \ref{enumeration}, we will revisit the proof for counting the number of perfect matchings in truncated cubic graphs, obtained in~\cite{CiucuNotes10, CiucuLiuYang}. Then we will use that approach to describe how to efficiently sample a uniform perfect matching from it. We finish the section by describing the partial order on the space of perfect matchings. 

In Section \ref{decorated flip dynamics}, we consider the classes of graphs that are called \emph{decorated graphs}. They are obtained from an initial graph $G$ by replacing every vertex in it with a connected graph that satisfies certain properties regarding the existence of a perfect matching in its subgraphs. We prove a general upper bound for the cycles that assures that the decorated graph is flip-connected. Note that this bound depends not only on the sizes of basis cycles but also on the sizes of decorations. We then prove Theorems \ref{mainthm2} and \ref{mainthm3}.

\section{Enumeration of perfect matchings} \label{enumeration}

In this section, we recall the result of Ciucu, Liu, and Yang~\cite{CiucuLiuYang} on the enumeration of perfect matchings for truncated cubic graphs (equivalently, Fisher graphs obtained from cubic graphs).

For $k \in \mathbb{N}$, we define a $k$-factor of $G$ as a $k$-regular spanning subgraph of $G$. In particular, $1$-factors are precisely perfect matchings. More generally, a $\{k, \ell\}$-factor is a spanning subgraph in which every vertex has degree either $k$ or $\ell$. We denote by $F_{\{k, \ell\}}(G)$ the number of $\{k, \ell\}$-factors of $G$. 

Before stating their result, we describe a bijection between the $\{1,3\}$-factors of $G$ and the perfect matchings of its truncated graph $G_\Delta$, obtained by replacing each vertex of $G$ with a triangle.

\begin{figure}[ht]
    \centering
    \includegraphics[width=0.55\textwidth]{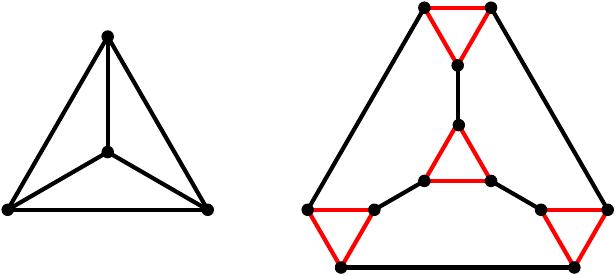}
    \caption{\label{fig:trunctetra} The tetrahedron $T$ and the truncated tetrahedron $T_\Delta$.}
\end{figure}

\begin{lem} [Proposition 5.~\cite{CiucuLiuYang}] \label{perf matchings and 1-3 factors}
    Let $G$ be a cubic graph. There is a bijection between $\{1,3\}$-factors of $G$ and perfect matchings of the truncated cubic graph $G_\Delta$. Therefore,
    \[M(G_{\Delta}) = F_{\{1,3\}}(G).\]
\end{lem}

In the following lemma, we describe how the number of $\{1,3\}$-factors doubles once we replace one of the vertices of a cubic graph with a triangle.

\begin{lem} \label{2 x 1,3 factors}
    Let $G$ be a cubic graph and $v \in V(G)$. Then
    \[F_{\{1,3\}}(G_{\Delta_v}) = 2 F_{\{1,3\}}(G),\]
    where $G_{\Delta_v}$ is the graph obtained from $G$ by replacing the vertex $v$ with a triangle.
\end{lem}

\begin{proof}
    Consider a $\{1,3\}$-factor $F$ in $G$. There are two cases, depending on the degree of $v$ in $F$.
    
    If $v$ has degree $1$, then there are exactly two ways to obtain a $\{1,3\}$-factor in $G_{\Delta_v}$. We can do this either by taking an edge on the remaining two vertices of the triangle or by taking its complement within the triangle. See the left part of Figure \ref{fig:2times13}. If $v$ has degree $3$, then we can take either all of the edges in the triangle or none of them. See the right part of Figure \ref{fig:2times13}.

    Thus, every $\{1,3\}$-factor of $G$ gives rise to exactly two $\{1,3\}$-factors of $G_{\Delta_v}$, and conversely every $\{1,3\}$-factor of $G_{\Delta_v}$ arises in this way. 
\end{proof}

\begin{figure}[ht]
    \centering
    \includegraphics[width=0.8\textwidth]{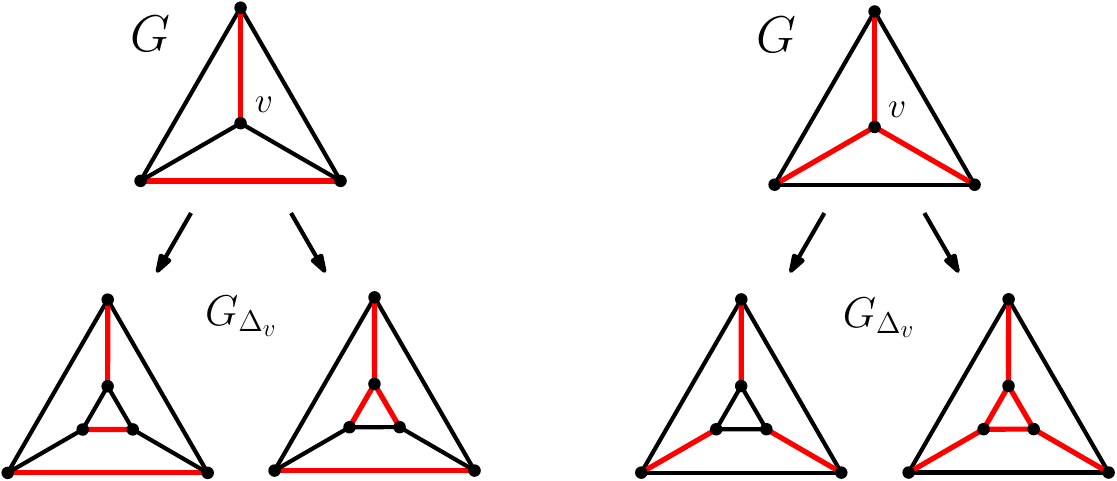}
    \caption{\label{fig:2times13} The graph $G_{\Delta_v}$ obtained by replacing the vertex $v$ of $G$ with a triangle. On the left (resp. right), the vertex $v$ has degree $1$ (resp. $3$) in a $\{1,3\}$-factor $F$. For each case, the two possibilities for $\{1,3\}$-factors in $G_{\Delta_v}$ are shown.}
\end{figure}

\begin{rem} \label{bijection 1-3 2-reg}
    Let $G$ be a cubic graph. There is a bijection between the set of $\{1,3\}$-factors and the set of $\{0,2\}$-factors of $G$. This can be seen by taking the complement of a $\{1,3\}$-factor in $G$ and noticing that every vertex must have degree either $0$ or $2$ in the resulting subgraph.
\end{rem}

We now state the following well-known result in graph theory. For completeness, we also include a short proof here.

\begin{lem} [~\cite{Biggs}, Chapters 4 and 5] \label{even subgraphs}
    Let $G$ be a graph on $n$ vertices, $m$ edges and $c$ connected components. Then the number of even subgraphs of $G$ is equal to $2^{m-n+c}$.
\end{lem}

\begin{proof}
    We choose a spanning forest $F$ such that it contains a spanning tree of every connected component of $G$. There are $n-c$ edges in $F$. Therefore, there are $m-n+c$ edges not contained in the spanning forest. Each edge not contained in $F$, when added to $F$, creates a unique cycle. These cycles then form the basis of the cycle space of $G$, whose dimension is $m-n+c$. As each even subgraph of $G$ is obtained as a symmetric difference over any subset of basis cycles, we obtain that there are $2^{m-n+c}$ even subgraphs of $G$.
\end{proof}

Using the fact that for cubic graphs $m = \frac{3}{2}n$, we obtain the following statement.

\begin{cor} \label{num of 2-reg}
    Let $G$ be a cubic graph. Then, the number of $\{1,3\}$-factors of $G$ is equal to $2^{\frac{n}{2}+c}$, where $n$ is the number of vertices and $c$ is the number of connected components in $G$.
\end{cor}

We conclude this section by stating the result of Ciucu, Liu, and Yang~\cite{CiucuLiuYang}.

\begin{thm} \label{perf_matchings of G_Delta}~\cite{CiucuLiuYang}
    Let $G_\Delta$ be a truncated cubic graph. Then
    \[M(G_\Delta) = 2^{\frac{n_\Delta}{6}+c},\]
    where $n_\Delta$ denotes the number of vertices of $G_\Delta$ and $c$ denotes the number of its connected components.
\end{thm}

\begin{proof}
    It suffices to put together the results from above. Using Remark \ref{bijection 1-3 2-reg} and Corollary \ref{num of 2-reg} we get that $F_{\{1,3\}}(G) = 2^{\frac{n}{2}+c}$. Then, by Lemma \ref{perf matchings and 1-3 factors}, we find that
    \[M(G_\Delta) = F_{\{1,3\}}(G) = 2^{\frac{n}{2}+c} = 2^{\frac{n_\Delta}{6}+c},\]
    where we used $n_\Delta = 3 n$.
\end{proof}

They used a different approach in~\cite{CiucuLiuYang} to prove Theorem \ref{perf_matchings of G_Delta}, namely the high temperature expansion of the Ising model. However, in the lecture notes by Ciucu~\cite{CiucuNotes10}, the method shown above was used to obtain the same result.

This exact formula, in fact, shows that any truncated cubic graph has an exponential number of perfect matchings. The Lovász-Plummer conjecture (1986)~\cite{LovaszPlummer} states that every cubic bridgeless graph (no edge disconnects the graph) has an exponential number of perfect matchings. This conjecture was first proven in the planar case~\cite{ChudnovskySeymour} and later in full generality~\cite{EKKKN}.

A similar construction to truncation was studied in~\cite{CygPilSkrek}, where they looked at the Klee graphs, which are defined as graphs that can be obtained from $K_4$ by recursively replacing a vertex with a triangle. For these graphs, they provide a lower bound on the number of perfect matchings. 

\subsection{Sampling algorithm and partition function}

We discuss how this approach to counting the number of perfect matchings in truncated cubic graphs can be used to generate a uniform perfect matching in linear time. Then we also go beyond the uniform case by fixing certain weights of the graph to be equal to $1$, while the remaining weights can take any positive real values.

Elkies, Kuperberg, Larsen, and Propp~\cite{ElkKupLarPropp} initiated the study of domino tilings of the Aztec diamond. They found a sampling algorithm that generates a uniform domino tiling. Since then, studying the Aztec diamond has become one of the most widely studied examples in the dimer model. This family of graphs is bipartite and planar. In this section, we present an algorithm for sampling uniform perfect matchings in non-bipartite graphs, including various non-planar graphs. This may create a new framework to explore similar questions to those previously asked about the Aztec diamond.

We proceed to explain the algorithm and use it to describe the partition function of truncated cubic graphs.

\subsubsection*{\textbf{Algorithm.}} Let $G$ be a cubic graph and denote by $G_\Delta$ its truncated cubic graph. Following the lines of the proof of Theorem \ref{perf_matchings of G_Delta}, we start by choosing a spanning forest $F$ such that it contains a spanning tree of every connected component of $G$. This can be done in linear time for cubic graphs, using, for example, the breadth-first search algorithm~\cite{Algorithms}. Then, every edge $e \in E(G) \setminus E(F)$ creates a unique cycle in $G$. This creates a cycle basis \[\mathcal{C} = \{C_1, C_2, \dots, C_k\}\] of the graph $G$, such that $k = \frac{n}{2}+c$, $n = |V(G)|$, and $c$ the number of its connected components. 

Then, the algorithm simply consists of choosing a subset $S \subseteq \mathcal{C}$ such that for all $1 \leq i \leq k$ we choose independently,
\[\mathbb{P}(C_i \in S) = \frac{1}{2}.\]
In other words, we are choosing a subset $S \subseteq \mathcal{C}$ uniformly at random among all the $2^k$ such subsets.

By taking the symmetric difference of the cycles in $S$, we obtain a uniform $\{0,2\}$-factor of $G$, which can be mapped using Remark \ref{bijection 1-3 2-reg} to a uniform $\{1,3\}$-factor in $G$. Using the bijection from Lemma \ref{perf matchings and 1-3 factors}, we finally obtain a uniform perfect matching on the truncated graph $G_\Delta$.

We can now apply this algorithm to obtain a specific partition function for this family of graphs.

\subsubsection*{\textbf{Partition function.}} Let $\mathcal{C} = \{C_1, C_2, \dots, C_k\}$ be a cycle basis of a cubic graph $G$, where $k = \frac{n}{2}+c$. Then we can define a weight function $\mathrm{wt}: \mathcal{C} \rightarrow \mathbb{R}_+$. We denote by $g: 2^{\mathcal{C}} \rightarrow \mathcal{M}(G_\Delta)$ the bijection between $\{0,2\}$-factors of $G$ and perfect matchings of $G_\Delta$.
Instead of considering a uniform $\{0,2\}$-factor as before, we can use the weight function on the cycles to define the following probability measure. 

\begin{defi}
    Let $G$ be a cubic graph, and let $S \subseteq \mathcal{C}$. Then, for every $1 \leq i \leq k$ we choose independently,
    \[\mathbb{P}(S) = \frac{\prod_{C \in S}\mathrm{wt}(C)}{Z_2(G)},\]
    where $Z_2(G)$ is the partition function for $\{0,2\}$-factors of $G$.
\end{defi}

\begin{prop}
    Let $G$ be a cubic graph. Then,
    \[Z_2(G) = \prod_{i = 1}^k (1 + \mathrm{wt}(C_i)).\]
    Moreover, there exists a weight function $\nu: \mathcal{M}(G_\Delta) \rightarrow \mathbb{R}_{+}$ such that \[Z_2(G) = Z(G_\Delta) = \sum_{M \in \mathcal{M}(G_\Delta)} \nu(M).\]
\end{prop}

\begin{proof}
    We write \[Z_2(G) = \sum_{S \subseteq \mathcal{C}} \mathrm{wt}(S) = \sum_{S \subseteq \mathcal{C}} \prod_{C_i \in S}\mathrm{wt}(C_i) = \prod_{i = 1}^k (1 + \mathrm{wt}(C_i)).\]
    Using the bijection $g$, we obtain a particular dimer partition function on $G_\Delta$, where $\nu: \mathcal{M}(G_\Delta) \rightarrow \mathbb{R}_{+}$, such that \[\nu(M) = \mathrm{wt}(g^{-1}(M)) = \mathrm{wt}(S) = \prod_{1 \leq i \leq k, C_i \in S}\mathrm{wt}(C_i).\]

\end{proof}

We finish this section by providing some examples of graphs that may be interesting to study.

\begin{ex} (Klee graphs)
    Recall that a graph is called a Klee graph if it is obtained by replacing one of the vertices with a triangle from another Klee graph or from $K_4$. Similarly to what is shown in Lemma \ref{2 x 1,3 factors}, we can describe the partition function for $\{1,3\}$-factors of Klee graphs, where at each step we introduce a new triangle such that two of its edges have a weight of $1$, and the third can have an arbitrary positive real weight $w$. It is easy to check that their partition functions can be related by:
    \[Z_{\{1,3\}}(G_{\Delta_v}) = (1+w) Z_{\{1,3\}}(G),\]
    where $G_{\Delta_v}$ is obtained from $G$ by replacing the vertex $v$ with a triangle. 
\end{ex}

Another possible variant would be to study \emph{random Klee graphs}, which are defined as Klee graphs for which, at each step, we decide uniformly at random which vertex should be replaced by a triangle.

\begin{ex} (Truncating the tetrahedron)
We can also consider another geometrically interesting object, which is, in fact, a Klee graph itself. We can define a sequence of graphs:
\[T^{n+1} = T^n_{\Delta},\]
where $n \geq 0$, and $T^0$ is the tetrahedron. This can be viewed in $3$-dimensional space as starting with a tetrahedron and truncating its corners at each step. The first step of this procedure is shown in Figure \ref{fig:trunctetra}. 
\end{ex}

\subsection{Partial order on perfect matchings} \label{poset_perf_match}

To study the flip dynamics, we first introduce a partial order on perfect matchings. We begin by recalling the bijection between the $\{0,2\}$-factors of $G$ and the perfect matchings of $G_\Delta$, established by Lemma \ref{perf matchings and 1-3 factors} and Remark \ref{bijection 1-3 2-reg}. 

We denote this bijection by $g: 2^\mathcal{C} \rightarrow \mathcal{M}(G_\Delta)$, where $\mathcal{C} = \{C_1, \dots, C_k\}$ denotes a cycle basis for the cubic graph $G$. The set $2^\mathcal{C}$ is naturally equipped with the partial order given by inclusion. In Figure \ref{fig:partialorderloops}, we can see the desired poset on the tetrahedron.

\begin{figure}[ht]
    \centering
    \includegraphics[width=0.5\textwidth]{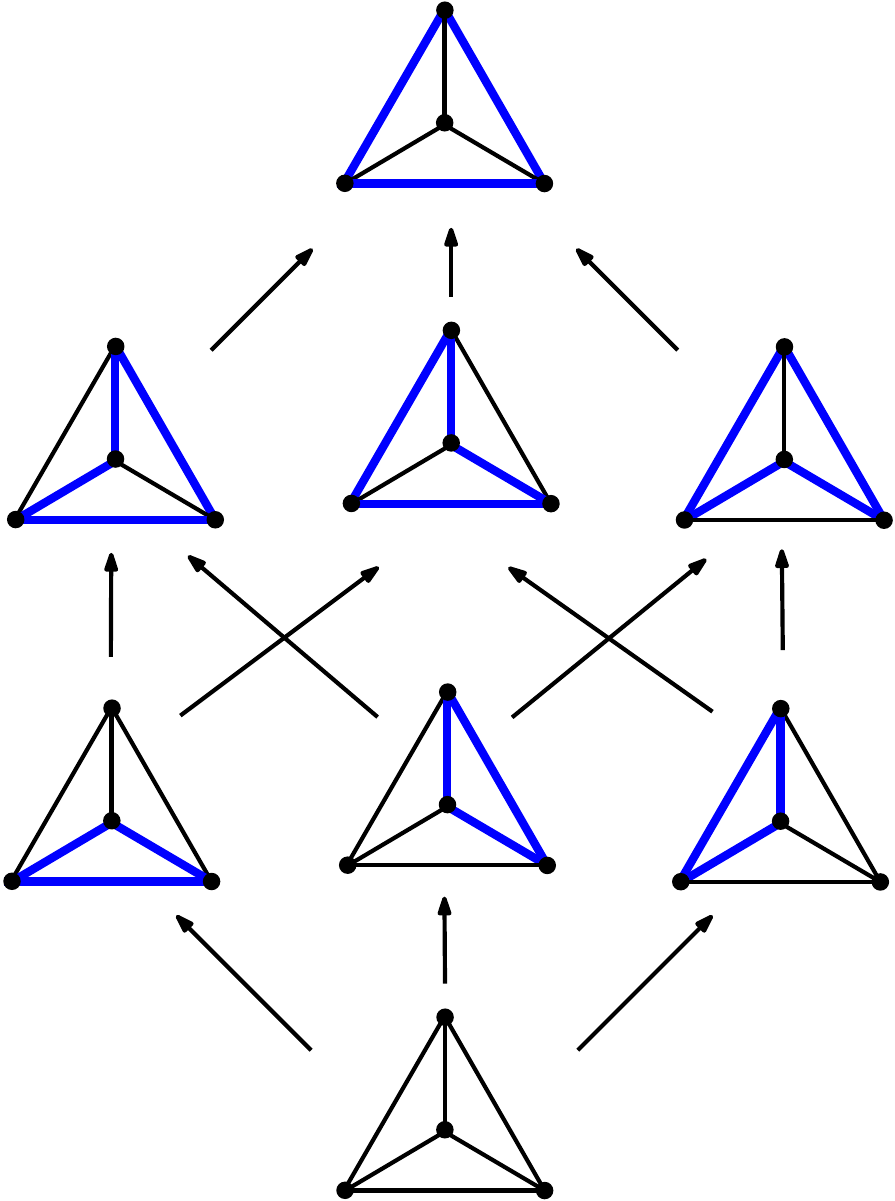}
    \caption{\label{fig:partialorderloops} The partial order on $2$-regular subgraphs of the tetrahedron.}
\end{figure}

Using the bijection $g$, we can define the poset on $\mathcal{M}(G_\Delta)$. That is, \[M_1 = g(S_1) \leq g(S_2) = M_2\] if and only if $S_1 \subseteq S_2$, where $M_1, M_2 \in \mathcal{M}(G_\Delta)$ and $S_1, S_2 \in 2^\mathcal{C}$. Figure \ref{fig:partialorderperfmatchings} illustrates this partial order for the truncated tetrahedron.

\begin{figure}[ht]
    \centering
    \includegraphics[width=0.55\textwidth]{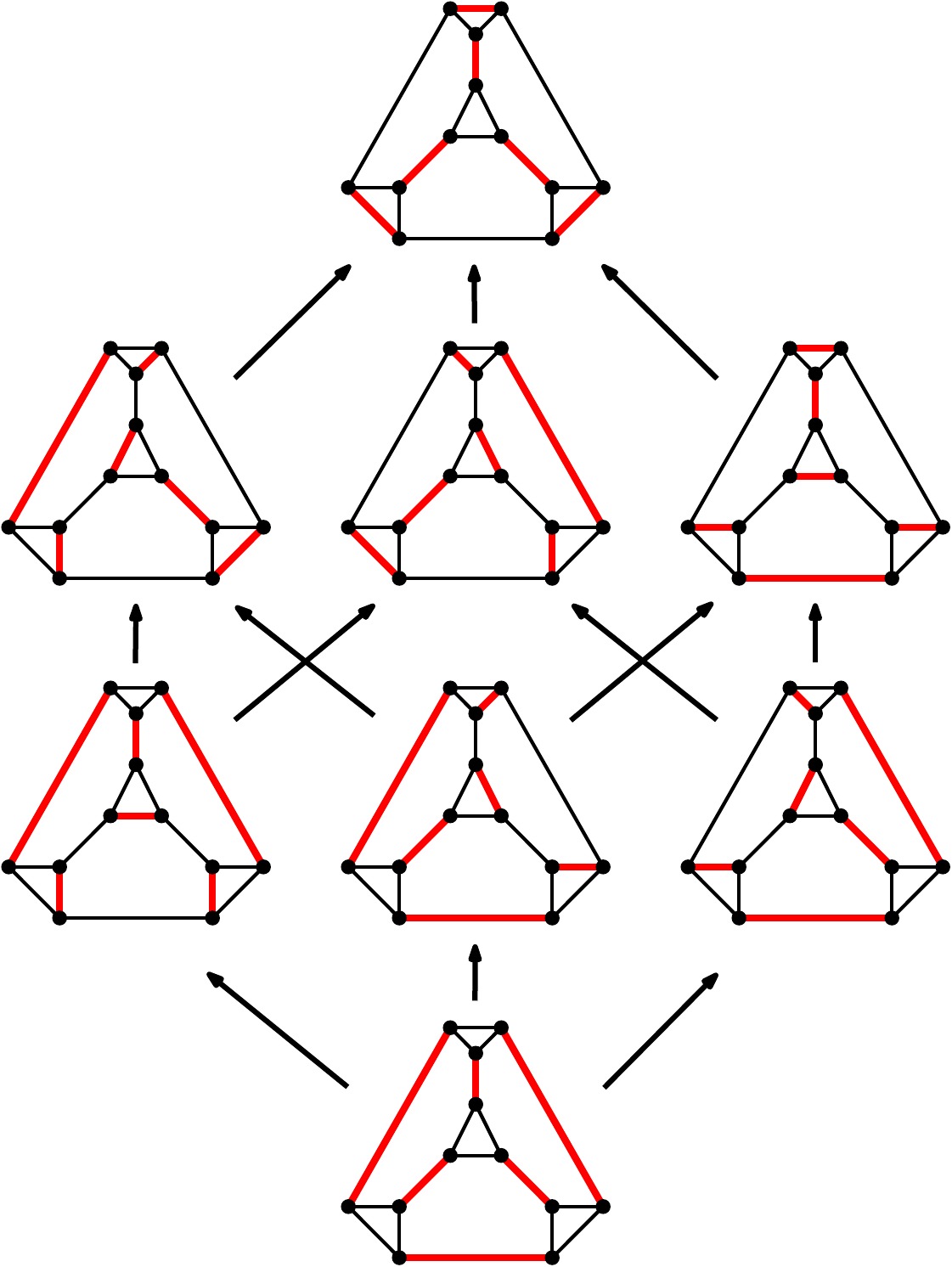}
    \caption{\label{fig:partialorderperfmatchings} The partial order on perfect matchings of the truncated tetrahedron.}
\end{figure}

\subsubsection*{\textbf{The non-planar case. }}

If a cubic graph $G$ is planar, after fixing its embedding in the plane, the natural choice for a cycle basis is the set of its inner faces. If $G$ is non-planar, the situation is more subtle, as there are many choices for its cycle basis. As we have already described, any spanning forest with a single tree per connected component of $G$ gives a cycle basis. In some cases, we might find that some of the cycles in the basis are rather long, possibly visiting almost all the vertices of $G$. This is the case that we would like to avoid, as it only provides a trivial upper bound, which can also be obtained by simply superimposing two perfect matchings.

The problem of finding a \emph{minimum cycle basis} has been extensively studied in the computer science literature. They include edge-weighted and unweighted versions, where different algorithms have been found to efficiently obtain a minimum cycle basis of a given graph. See, for example~\cite{MehlhornMichail}. In this paper, we do not make use of these algorithms, we only point out their existence.

\section{Flip dynamics for decorated graphs} \label{decorated flip dynamics}

In this section, we study the flip dynamics on a family of graphs that we call \emph{decorated graphs}. These graphs are obtained by replacing each vertex of a graph with another connected graph called a \emph{decoration}. We denote a decoration by $(H, D)$, where $H$ is a simple connected graph and $D \subseteq V(H)$ is the set of \emph{distinguished} vertices. 

\begin{defi}
    We say that a decoration $(H, D)$ is \emph{acceptable} if, for every even subset of vertices $S \subseteq D$, the graph $H \setminus S$ contains a perfect matching.
\end{defi}

If $D = V(H)$ and we take $S = D$ in the definition, then $H \setminus S$ has a perfect matching, which is an empty set. From this, we also observe that $H$ should have an even number of vertices. On the other hand, $D$ can also be odd for $(H,D)$ to be acceptable.

Some of the most natural examples of acceptable decorations are complete graphs $K_{2d}$. As a non-example, consider the cycle on $4$ edges, since by removing two non-adjacent vertices, the rest of the graph does not contain a perfect matching.

We denote by $\mathcal{F}_{\text{dec}}$ the family of all acceptable decorations.

\begin{lem} \label{H+e acceptable}
    Let $(H,D)$ be an acceptable decoration.
    \begin{enumerate}
        \item If $|D| > 2$, then $H$ is a non-bipartite graph.
        \item if $v_1, v_2 \in V(H)$ such that $v_1v_2 \notin E(H)$, then $(H', D)$, where $H$ is obtained by adding the edge $v_1v_2$ to $H$, is an acceptable decoration.
    \end{enumerate}
\end{lem}

\begin{proof}
    The first can be seen by supposing the opposite, that $H$ is bipartite, and coloring $n$ of its vertices black and the remaining $n$ vertices white. Since $|D| > 2$, we can find in $D$ at least two vertices of the same color. If we remove them, the subgraph of $F$ does not contain a perfect matching.

    The second statement follows because if $H \setminus S$ contains a perfect matching for an even subset $S \subseteq D$, then $H' \setminus S$ also contains a perfect matching.
\end{proof}

\begin{defi} \label{deco_graph}
    Let $G$ be a graph. Let $\mathcal{D} : V(G) \to \mathcal{F}_{\mathrm{dec}}$ be a function assigning to each vertex $v \in V(G)$ an acceptable decoration $\mathcal{D}(v) = (H_v, D_v)$, where $H_v$ is a graph and $D_v \subseteq V(H_v)$ satisfies $|D_v| = \deg_G(v)$. 

    We define $\mathcal{F}_{(G,\mathcal{D})}$ to be the family of \emph{decorated graphs} $G_{\mathcal{D}}$ constructed as follows: for each $v \in V(G)$, replace $v$ with the graph $H_v$, and for each edge $e = uv \in E(G)$, add an edge $\phi(e) = u'v'$, where $u' \in D_u$ and $v' \in D_v$. We require that the map $\phi : E(G) \to E(G_{\mathcal{D}})$ has a vertex-disjoint image, that is, no two edges $\phi(e)$ share a common endpoint.
\end{defi}

We write $G_{\mathcal{D}} \in \mathcal{F}_{ (G, \mathcal{D})}$ when considering decorated graphs for an arbitrary choice of $\phi$ satisfying the above conditions. In case we want a particular map, we denote the corresponding graph by $G^{\phi}_{\mathcal{D}}$.

Note that $\mathcal{F}_{ (G, \mathcal{D})}$ is indeed a family of graphs because there are usually multiple ways to decorate a vertex, depending on the choice of the cyclic permutation that decides how the distinguished vertices are mapped to the neighbors of a given vertex $v \in V(G)$. Therefore, we consider decorated graphs arising from all possible choices of cyclic permutations.

\begin{defi}
    Let $G_\mathcal{D} \in \mathcal{F}_{ (G, \mathcal{D})}$ and let $\mathcal{C} = \{C_1, \dots, C_k\}$ be a cycle basis of $G$. For a decoration $\mathcal{D}(v)$ at a vertex $v \in V(G)$, a perfect matching in $G_{\mathcal{D}}$ matches an even number of vertices internally and an even number of vertices outside the decoration.
    
    We define the map $t: \mathcal{M}(G_{\mathcal{D}}) \rightarrow 2^\mathcal{C}$ as follows. For every $M \in \mathcal{M}(G_{\mathcal{D}})$, the image $t(M)$ is the even subgraph of $G$ obtained by collapsing all decorations of $G_\mathcal{D}$.
\end{defi}

Recall from Lemma \ref{even subgraphs} that every even subgraph, as we discussed in Section \ref{enumeration}, can be described as a subset of a cycle basis $\mathcal{C}$. 

\begin{lem} \label{t_surj}
    The map $t: \mathcal{M}(G_{\mathcal{D}}) \rightarrow 2^\mathcal{C}$ is surjective. 
\end{lem}

\begin{proof}
    This follows from the fact that for every $v \in V(G)$, the decoration $\mathcal{D}(v)$ is acceptable. Therefore, for any choice of an even subgraph in $G$, we can find a perfect matching in $G_{\mathcal{D}}$.
\end{proof}

Similarly to the case of truncated cubic graphs, the map $t$ defines a poset structure on the long edges of $G_{\mathcal{D}}$. We can now proceed to prove the following result about the flip dynamics.

\begin{thm} \label{general_flip_theorem}
    Let $G$ be a graph, and denote by $G_\mathcal{D} \in \mathcal{F}_{ (G, \mathcal{D})}$ the decorated graph. The space of perfect matchings $\mathcal{M}(G_\mathcal{D})$ is flip-connected on the cycles of length at most 
    \[\max_{1 \leq i \leq k}\Big\{\sum_{v \in V(C_i)} |\mathcal{D}(v)|\Big\},\]
    where $\{C_1, \dots, C_k\}$ forms a cycle basis of $G$, and $|\mathcal{D}(v)|$ is the number of vertices of the graph that replaces the vertex $v \in V(G)$ in $G_\mathcal{D}$. 
\end{thm}

In order to prove this theorem, we need the following result.

\begin{lem} \label{unique path}
    Let $G$ be a connected graph and let $u, v \in V(G)$ be such that $G \setminus u$ and $G \setminus v$ both contain a perfect matching. Then there is a unique path from $u$ to $v$ in $M_1 \cup M_2$, where $M_1 \in \mathcal{M}(G \setminus u)$ and $M_2 \in \mathcal{M}(G \setminus v)$.
\end{lem}

\begin{proof} 
    The result follows directly by noticing that each vertex $w \in V(G) \setminus \{u, v\}$ has degree two in $M_1 \cup M_2$, while $u$ and $v$ have degree one. Therefore, there is only one path in $M_1 \cup M_2$, and it is exactly the path connecting $u$ and $v$. See Figure \ref{fig:ddsquaregrid} for an example.
\end{proof}

\begin{figure}[ht]
    \centering
    \includegraphics[width=0.4\textwidth]{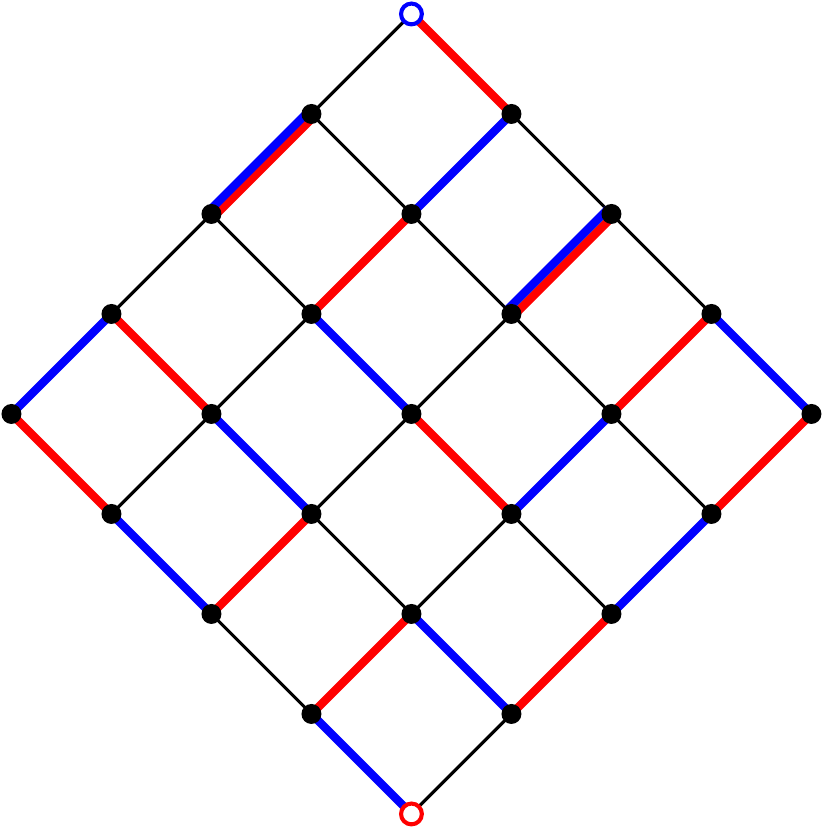}
    \caption{\label{fig:ddsquaregrid} A perfect matching without the top vertex superimposed by a perfect matching without the bottom vertex.}
\end{figure}

\begin{proof} [Proof of Theorem \ref{general_flip_theorem}]
    Let $M \in \mathcal{M}(G_\Delta)$, such that $t(M) = \emptyset$. We want to show that by performing flips, we can get from $M$ to any other perfect matching $M'$. This will be enough because the flip operation is reversible. Such a perfect matching $M$ exists because by Lemma \ref{t_surj} the map $t$ is surjective. This corresponds to the perfect matching obtained by choosing all the long edges in $G_{\mathcal{D}}$ (no edge comes from any of the decorations). 

    To prove this claim, we use the poset structure described in Section \ref{poset_perf_match}. Let $M_1, M_2 \in \mathcal{M}(G_{\mathcal{D}})$ and $S_1, S_2 \in 2^\mathcal{C}$, such that $t(M_1) = S_1$ and $t(M_2) = S_2$. It suffices to show that if $S_1 \subseteq S_2$ and $S_2 \setminus S_1 \subseteq \{C_\}$, for some $C_i \in \mathcal{C}$, then $M_2$ can be obtained from $M_1$ by performing flips on decorations independently and a single flip on the certain cycle of $G_{\mathcal{D}}$. By performing this flip operation consecutively, we can obtain $M'$ from $M$.
    
    If $S_1 = S_2$, then we need to perform the flips independently on each decoration $\mathcal{D}(v)$. We use a trivial upper bound for the length of the cycles involved in these flips, which is obtained by superimposing any two perfect matchings of $\mathcal{D}(v)$ and is given by $|\mathcal{D}(v)|$ (when a cycle visits all the vertices of $\mathcal{D}(v)$. This bound is smaller than the upper bound we want to prove. 

    We assume now that $S_1 \neq S_2$, it suffices to consider the case when $S_2 \setminus S_1 = \{C_i\}$, for some $1 \leq i \leq k$. Let $v \in V(C_i)$ and $\mathcal{D}(v)$ be the corresponding decoration. Note that the long edges in $M_1 \triangle M_2$ are exactly the edges corresponding to $C_i$. Moreover, there are two distinguished vertices of $\mathcal{D}(v)$ that are attached to the two long edges corresponding to the two edges of $C_i$ that pass through $v$. 
    
    By Lemma \ref{unique path}, there is a unique path passing through $\mathcal{D}(v)$ that connects these two vertices. The length of this path is bounded from above by $|\mathcal{D}(v)| - 1$. Thus, the maximum size of a cycle on which a flip needs to be performed to get from $M_1$ to $M_2$ such that $S_2 \setminus S_1 = \{C_i\}$ is equal to \[|V(C_i)| + \sum_{v \in V(C_i)}(|\mathcal{D}(v)|-1) = \sum_{v \in V(C_i)}|\mathcal{D}(v)|,\]
    where $|V(C_i)|$ appears in the sum because, besides the paths through every $\mathcal{D}(v)$, we need to take into account also the edges of $C_i$. By taking the maximum over all $i \in \{1, \dots, k\}$, we obtain the desired upper bound.
\end{proof}

\subsection{Clique-decorated graphs}

Here, we discuss the case when each decoration is assumed to be a complete graph. We assume that the initial graph $G$ has all even degrees. We call such a decoration a clique decoration and denote it by $G_{clique}$.

It is simple to notice that the decoration $(K_d, V(K_d))$, for even $d$, is acceptable, as after deleting an even number of vertices, we obtain a complete graph that always has a perfect matching.

\begin{lem} \label{4-flips}
    Let $K_{2d}$ be a complete graph on $2d$ vertices for $d \geq 2$. Then $\mathcal{M}(K_{2d})$ is flip-connected on the cycles of length $4$.
\end{lem}

\begin{proof}
    We proceed by induction on $d$. If $d = 2$, then the result follows, as two distinct perfect matchings of $K_4$ always form a cycle of length $4$. 

    We suppose the statement is true for $d \geq 2$. Then, let $M_1, M_2 \in \mathcal{M}(K_{2d+2})$. If $M_1 \cup M_2$ forms more than one cycle in $K_{2d+2}$, then we are done by the inductive hypothesis, because each of them will have a length of at most $2d$. 
    
    It remains to consider the case when $M_1 \cup M_2$ consists of a single cycle. Let $v_1, v_2, v_3, v_4 \in V(K_{2d+2})$ be such that $v_1v_2, v_3v_4 \in M_1$ and $v_2v_3 \in M_2$. We then perform a flip on the cycle defined by these four vertices to replace the edges of $M_1$ with the edges $v_2v_3$ and $v_4v_1$. Therefore, after performing this flip, we obtain in $K_{2d+2}$ a double edge on $v_2v_3$ and a cycle of length $2d$. See Figure \ref{fig:K6flips}. Using the inductive hypothesis, the perfect matchings forming the cycle of length $2d$ are related by flips on cycles of length $4$.
\end{proof}

\begin{figure}[ht]
    \centering
    \includegraphics[width=0.8\textwidth]{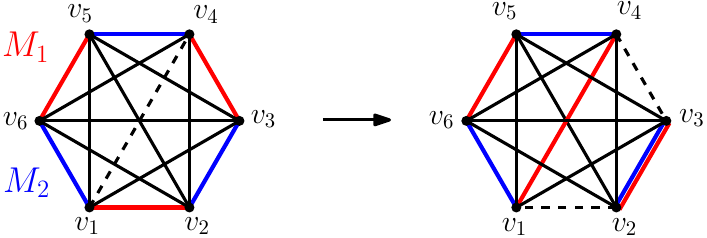}
    \caption{\label{fig:K6flips} A cycle of length $6$ deformed into a cycle of length $4$ and a double edge on $v_2v_3$.}
\end{figure}

\begin{proof} [Proof of Theorem \ref{mainthm2}]
    We first assume that the vertices of $G$ have even degrees. 
    
    Let $M_1, M_2 \in \mathcal{M}(G_{clique})$. Let $S_1$ and $S_2$ be the corresponding even subgraphs in $G$. If $S_1 = S_2$, then it suffices to perform flips on each clique decoration separately. By Lemma \ref{4-flips}, on a complete graph with at least $4$ vertices, it suffices to use the cycles of length $4$.

    Now, we assume that $S_1 \neq S_2$. Similarly to Theorem \ref{general_flip_theorem}, it is enough to assume that $S_2 \setminus S_1 = \{C_i\}$, for some $1 \leq i \leq k$. Let $v \in V(C_i)$ and denote by $u_1$ and $u_2$ its neighbors in $C_i$. We denote by $K_v$ the complete graph on $\deg(v)$ vertices, which replaces $v$ in $G_{clique}$. There are two possibilities:
    \begin{enumerate}
        \item $u_1v, u_2v \in E(S_1)$ or $u_1v, u_2v \notin E(S_1)$.
        This corresponds to the left side of Figure \ref{fig:clique2cases}. It suffices to complete the perfect matching $M_1$ into a cycle by extending it with a single edge in $K_v$ (denoted by $e_1$ in Figure \ref{fig:clique2cases}, left). One can obtain a perfect matching containing this edge by applying to $M_2$ a sequence of flips internal to $K_v$.
        \item $u_1v \in E(S_1)$ and $u_2v \notin E(S_1)$ (or the opposite).
        In this case, there are exactly two long edges in $M_1 \triangle M_2$ such that exactly one of their end vertices is in the decoration $K_v$. See the right side of Figure \ref{fig:clique2cases}. We can complete the cycle by connecting these two vertices with a path of length $2$ that visits a vertex that is not visited by a double edge in $M_1 \cup M_2$. Such a vertex always exists because we assumed that $G$ has all even degrees, and $M_1 \triangle M_2$ visits an odd number of vertices of $K_v$. 
    \end{enumerate}
    
    The length of the longest cycle on which a flip is performed is maximized if every $v \in V(C_i)$ satisfies the second condition. In that case, the length of the cycle is bounded by $3|C_i|$, since to each edge of $C_i$, we need to assign a $2$-path in $K_v$, for all $v \in V(C_i)$.

    The case when the vertices of $G$ have odd degrees is completely analogous. Instead of even subgraphs, we need to take odd subgraphs $S_1$ and $S_2$ that correspond to perfect matchings $M_1$ and $M_2$. Then we need to repeat the same case analysis as before.
\end{proof}

\begin{figure}[ht]
    \centering
    \includegraphics[width=0.75\textwidth]{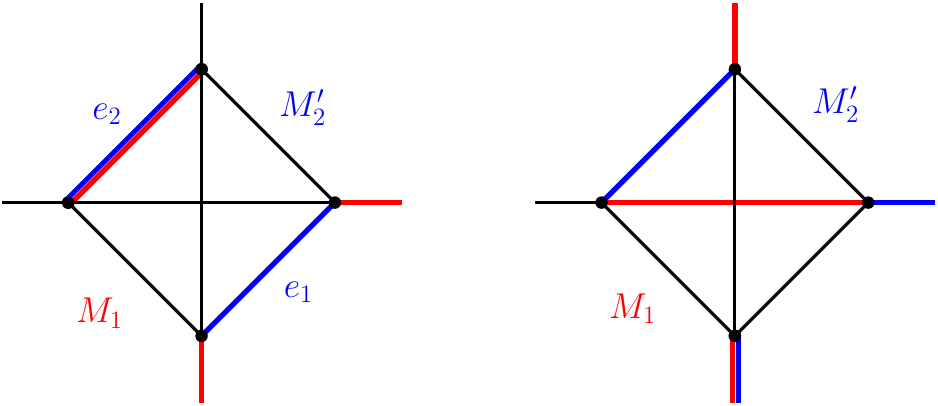}
    \caption{\label{fig:clique2cases} Two possible cases for performing flips depending on $E(C_i) \cap E(S_1)$. The perfect matching $M'_2$ is obtained from $M_2$, by a sequence of flips within $K_v$.}
\end{figure}

Observe that Theorem \ref{mainthm2} provides a stronger upper bound than Theorem \ref{general_flip_theorem} because the lengths of the cycles depend only on the cycle basis and not on the size of the inserted decorations.

\subsection{Fisher graphs}

As another important example of an acceptable decoration, we study the well-known Fisher decoration.  
See Figure \ref{fig:fishdeg5} for an example. Introduced by Fisher~\cite{Fisher66}, this decoration was used to establish a correspondence between the Ising model on a graph $G$ and the dimer model on the graph $F_G$, obtained by applying the Fisher decoration to every vertex of $G$.

It is straightforward to check that the Fisher decoration is acceptable. Indeed, after removing an even number of distinguished vertices, we always obtain a graph that contains exactly two perfect matchings. Using this observation and the same ideas as in the proof of Theorem \ref{perf_matchings of G_Delta}, we can count the number of perfect matchings of Fisher graphs.

\begin{lem}
    Let $G$ be a connected graph on $m$ edges and minimum degree at least $3$, and let $F_G$ be its Fisher graph. Then \[M(F_G) = 2^{m+1}.\]
\end{lem}

\begin{proof}
    By Lemma \ref{even subgraphs}, there are $2^{m-n+1}$ even subgraphs in $G$, where $n$ is the number of vertices of $G$. Each even subgraph of $G$ gives rise to $2^n$ perfect matchings of $F_G$, since every subgraph of a Fisher decoration obtained by removing an even number of distinguished vertices has exactly two perfect matchings. In total, we get \[M(F_G) = 2^{m-n+1}2^n = 2^{m+1}.\]
\end{proof}

Using the approach we described in Theorem \ref{general_flip_theorem}, as well as a particular structure of Fisher graphs, we can provide more precise upper bounds for the cycles that ensure flip-connectivity. 

\begin{proof} [Proof of Theorem \ref{mainthm3}]
    Let $M_1, M_2 \in \mathcal{M}(F_G)$, we can distinguish two cases.
    \begin{enumerate}
        \item The corresponding even subgraphs of $M_1$ and $M_2$ in $G$ are the same.

        In this case, to go from $M_1$ to $M_2$, it suffices to perform flips on cycles that are contained in every Fisher decoration. The length of these cycles is at most 
        \[\max_{1 \leq j \leq n} 2\deg(v_j),\]
        and the bound $2 \deg(v_j)$ is attained when the vertex $v_j$ is not included in the corresponding even subgraph.

        \item The even subgraphs $S_1$ (resp. $S_2$) associated with $M_1$ (resp. $M_2$) are not the same.

        As in Theorem \ref{general_flip_theorem}, it suffices to consider the case when $S_2 \setminus S_1 = \{F_i\}$, for some $1 \leq i \leq k$. Analogously to Theorem \ref{mainthm2}, we can have two possibilities. Let $v \in V(F_i)$  and let $u_1, u_2$ be the neighbors of $v$ in $F_i$.

        \begin{enumerate} [(a)]
            \item Suppose that either $u_1v, u_2v \in E(S_1)$ or $u_1v, u_2v \notin E(S_1)$. 
            
            By the planarity of $G$, and hence of $F_G$, the two distinguished vertices in the decoration of $v$ appear consecutively in the cyclic order. Then, by taking a $3$-path to visit the Fisher decoration associated with the vertex $v$, we get a cycle that extends the perfect matching $M_1$. The remaining edges of the $3$-path come from the perfect matching $M'_2$, which is obtained by applying the flips on $M_2$ within the Fisher decoration (their length is at most $2\deg(v)$). See the left side of Figure \ref{fig:fisher2cases}. 
            \item Suppose that $u_1v \in E(S_1)$ and $u_2v \notin E(S_1)$ (or vice versa). 
            
            The perfect matching $M_1$ can be extended using a $2$-path visiting the Fisher decoration, which can be seen in Figure \ref{fig:fisher2cases}. Similarly to the previous case, the perfect matching $M'_2$ is obtained from $M_2$ by flips within the Fisher decoration of $v$.
        \end{enumerate}

        In order to conclude the proof, we notice that the maximum length of a cycle is obtained if every vertex in $V(F_i)$ satisfies the condition (a), which means that every edge of $F_i$ is extended with a $3$-path in the corresponding Fisher decoration. In total, this gives a cycle whose length is at most $4 |F_i|$.
    \end{enumerate}
\end{proof}

\begin{figure}[ht]
    \centering
    \includegraphics[width=0.75\textwidth]{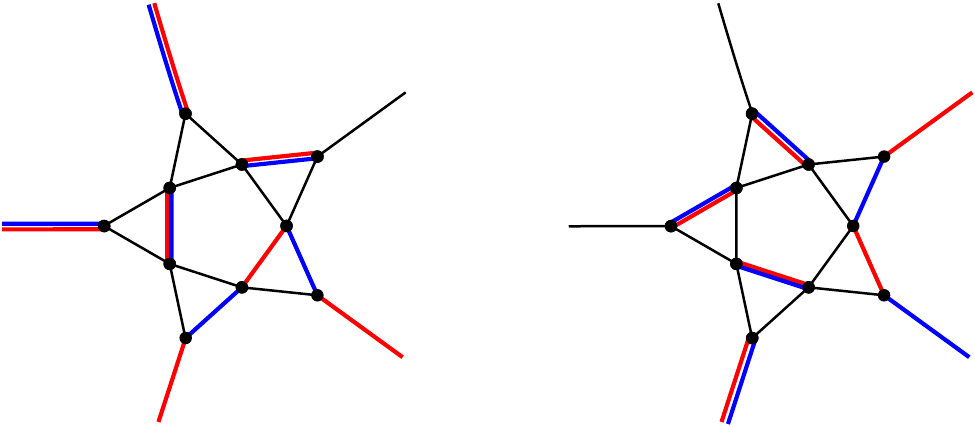}
    \caption{\label{fig:fisher2cases} Two possible cases for performing flips on Fisher decoration. The cycle on the left visits three edges of the decoration, while the cycle on the right visits two edges.}
\end{figure}

\begin{rem}
    By Lemma \ref{H+e acceptable}, we can add any missing edges to Fisher decorations, and they will still be acceptable decorations. Moreover, the same upper bound as in Theorem \ref{mainthm3} holds. In fact, adding new edges can only improve the upper bounds.
\end{rem}

We now briefly discuss the non-planar case. Recall that when we defined the Fisher graphs, we first fixed a drawing of $G$, then replaced every vertex with a Fisher decoration.

\begin{cor}
    Let $G$ be a non-planar graph on $n$ vertices, and denote by $F_G$ its Fisher graph. Then $\mathcal{M}(F_G)$ is flip-connected on the cycles of length at most
    \[\max_{1 \leq i \leq k} \left\{\sum_{v \in V(C_i)} (\deg(v)+2) \right\},\] where $\{C_1, \dots, C_k\}$ denotes a cycle basis of $G$.
\end{cor}

\begin{proof}
    If we apply Theorem \ref{general_flip_theorem} directly, then for each cycle $C_i$, the largest cycle on which a flip may be performed has length at most \[\sum_{v \in V(C_i)} 2\deg(v).\]
    This bound can be improved further. For any two distinguished vertices in a Fisher decoration of $v$, there are two internally edge-disjoint paths $P_1$ and $P_2$ joining them such that $|E(P_1)| + |E(P_2)| \leq 2\deg(v)$. Consequently, \[\min \{|E(P_1)|, |E(P_2)|\} \leq \deg(v).\] As in Figure \ref{fig:fisher2cases}, depending on the parity of the chosen path, we may need to extend it by one additional edge. Therefore, together with $|V(C_i)|$ edges from the original graph, we obtain
    \[|V(C_i)| + \sum_{v \in V(C_i)} (\deg(v)+1) = \sum_{v \in V(C_i)} (\deg(v)+2),\]
    and the result follows.
\end{proof}

\subsection*{Acknowledgements}

I am grateful to my PhD advisors, Cédric Boutillier and Sanjay Ramassamy, for their support and many valuable suggestions. I would also like to thank Seok Hyun Byun for pointing out the references~\cite{CiucuNotes10, CiucuLiuYang}, as well as Ivailo Hartarsky and Nedialko Bradinoff for insightful discussions.

\bibliographystyle{plainurl}
\bibliography{references}

@book {LovaszPlummer,
    AUTHOR = {Lov\'asz, L\'aszl\'o{} and Plummer, Michael D.},
     TITLE = {Matching theory},
      NOTE = {Corrected reprint of the 1986 original},
 PUBLISHER = {AMS Chelsea Publishing, Providence, RI},
      YEAR = {2009},
     PAGES = {xxxiv+554},
      ISBN = {978-0-8218-4759-6},
   MRCLASS = {90C10 (05B35 05C70 90C05)},
  MRNUMBER = {2536865},
}

@article {EKKKN,
    AUTHOR = {Esperet, Louis and Kardo\v{s}, Franti\v{s}ek and King, Andrew D.
              and Kr\'al, Daniel and Norine, Serguei},
     TITLE = {Exponentially many perfect matchings in cubic graphs},
   JOURNAL = {Adv. Math.},
  FJOURNAL = {Advances in Mathematics},
    VOLUME = {227},
      YEAR = {2011},
    NUMBER = {4},
     PAGES = {1646--1664},
      ISSN = {0001-8708,1090-2082},
   MRCLASS = {05C70 (05C30)},
  MRNUMBER = {2799808},
MRREVIEWER = {Vahan\ V.\ Mkrtchyan},
}

@article {ChudnovskySeymour,
    AUTHOR = {Chudnovsky, Maria and Seymour, Paul},
     TITLE = {Perfect matchings in planar cubic graphs},
   JOURNAL = {Combinatorica},
  FJOURNAL = {Combinatorica. An International Journal on Combinatorics and
              the Theory of Computing},
    VOLUME = {32},
      YEAR = {2012},
    NUMBER = {4},
     PAGES = {403--424},
      ISSN = {0209-9683,1439-6912},
   MRCLASS = {05C70 (05C10)},
  MRNUMBER = {2965284},
MRREVIEWER = {Stephan\ G.\ Wagner},
}

@article {Thurston90,
    AUTHOR = {Thurston, William P.},
     TITLE = {Conway's tiling groups},
   JOURNAL = {Amer. Math. Monthly},
  FJOURNAL = {American Mathematical Monthly},
    VOLUME = {97},
      YEAR = {1990},
    NUMBER = {8},
     PAGES = {757--773},
      ISSN = {0002-9890,1930-0972},
   MRCLASS = {52C20 (05B45 20F32 52C22)},
  MRNUMBER = {1072815},
MRREVIEWER = {Marjorie\ Senechal},
}

@article{HarLichTon2024,
  author       = {Ivailo Hartarsky and Lyuben Lichev and Fabio Lucio Toninelli},
  title        = {Local dimer dynamics in higher dimensions},
  journal      = {Annales de l'Institut Henri Poincaré, Comb. Phys. Interact.},
  year         = {2024},
  note         = {Published online first},
}

@article {CygPilSkrek,
    AUTHOR = {Cygan, Marek and Pilipczuk, Marcin and \v{S}krekovski, Riste},
     TITLE = {A bound on the number of perfect matchings in {K}lee-graphs},
   JOURNAL = {Discrete Math. Theor. Comput. Sci.},
  FJOURNAL = {Discrete Mathematics \& Theoretical Computer Science. DMTCS.},
    VOLUME = {15},
      YEAR = {2013},
    NUMBER = {1},
     PAGES = {37--52},
      ISSN = {1365-8050},
   MRCLASS = {05C70 (05C30)},
  MRNUMBER = {3022928},
MRREVIEWER = {Vahan\ V.\ Mkrtchyan},
}

@article {ElkKupLarPropp,
    AUTHOR = {Elkies, Noam and Kuperberg, Greg and Larsen, Michael and
              Propp, James},
     TITLE = {Alternating-sign matrices and domino tilings},
   JOURNAL = {J. Algebraic Combin.},
  FJOURNAL = {Journal of Algebraic Combinatorics. An International Journal},
    VOLUME = {1},
      YEAR = {1992},
    NUMBER = {2},
     PAGES = {111--132, 219--234},
      ISSN = {0925-9899,1572-9192},
   MRCLASS = {52C20 (05B45)},
  MRNUMBER = {1226347},
}

@book{Biggs, 
place={Cambridge}, 
edition={2}, 
series={Cambridge Mathematical Library}, 
title={Algebraic Graph Theory}, 
publisher={Cambridge University Press}, 
author={Biggs, Norman}, 
year={1974}, 
collection={Cambridge Mathematical Library},
}

@article{Fisher66,
    author = {Fisher, Michael E.},
    title = {On the Dimer Solution of Planar Ising Models},
    journal = {Journal of Mathematical Physics},
    volume = {7},
    number = {10},
    pages = {1776-1781},
    year = {1966},
    month = {10},
    issn = {0022-2488},
}

@article {KenRemila,
    AUTHOR = {Kenyon, Claire and R\'emila, Eric},
     TITLE = {Perfect matchings in the triangular lattice},
   JOURNAL = {Discrete Math.},
  FJOURNAL = {Discrete Mathematics},
    VOLUME = {152},
      YEAR = {1996},
    NUMBER = {1-3},
     PAGES = {191--210},
      ISSN = {0012-365X,1872-681X},
   MRCLASS = {05C70},
  MRNUMBER = {1388642},
MRREVIEWER = {Stanislav\ Jendrol'},
}

@article{RoisingZhang,
    author = "R{\o}ising, Henrik Schou and Zhang, Zhao",
    title = "{Ergodic Archimedean dimers}",
    primaryClass = "math-ph",
    journal = "SciPost Phys. Core",
    volume = "6",
    pages = "054",
    year = "2023"
}

@article {CiucuLiuYang,
    AUTHOR = {Ciucu, Mihai and Liu, Yan and Yang, Chunxia},
     TITLE = {Perfect matchings of {F}isher graphs of cubic graphs},
   JOURNAL = {Kyushu J. Math.},
  FJOURNAL = {Kyushu Journal of Mathematics},
    VOLUME = {66},
      YEAR = {2012},
    NUMBER = {2},
     PAGES = {291--302},
      ISSN = {1340-6116,1883-2032},
   MRCLASS = {05A15 (82B20)},
  MRNUMBER = {3051338},
}

@book {CiucuNotes10,
    AUTHOR = {Ciucu, Mihai},
     TITLE = {Perfect matchings and applications},
    SERIES = {COE Lecture Note},
    VOLUME = {26},
      NOTE = {Math-for-Industry (MI) Lecture Note Series},
 PUBLISHER = {Kyushu University, Faculty of Mathematics, Fukuoka},
      YEAR = {2010},
     PAGES = {iv+67},
   MRCLASS = {05-01 (05C70 82B23)},
  MRNUMBER = {2663560},
}

@article {MehlhornMichail,
    AUTHOR = {Mehlhorn, Kurt and Michail, Dimitrios},
     TITLE = {Minimum cycle bases: faster and simpler},
   JOURNAL = {ACM Trans. Algorithms},
  FJOURNAL = {ACM Transactions on Algorithms},
    VOLUME = {6},
      YEAR = {2010},
    NUMBER = {1},
     PAGES = {Art. 8, 13},
      ISSN = {1549-6325,1549-6333},
   MRCLASS = {05C85 (05C35 05C50 68R10 68W20 68W25)},
  MRNUMBER = {2654912},
MRREVIEWER = {Mohammed\ M.\ Jaradat},
}

@book {Algorithms,
    AUTHOR = {Cormen, Thomas H. and Leiserson, Charles E. and Rivest, Ronald
              L. and Stein, Clifford},
     TITLE = {Introduction to algorithms},
   EDITION = {Third},
 PUBLISHER = {MIT Press, Cambridge, MA},
      YEAR = {2009},
     PAGES = {xx+1292},
      ISBN = {978-0-262-03384-8},
   MRCLASS = {68-01 (05-01 05C85 68P05 68P10 68Q25 68Wxx)},
  MRNUMBER = {2572804},
}

@article {LaslierToninelli23,
    AUTHOR = {Laslier, Beno\^it and Toninelli, Fabio},
     TITLE = {The mixing time of the lozenge tiling {G}lauber dynamics},
   JOURNAL = {Ann. H. Lebesgue},
  FJOURNAL = {Annales Henri Lebesgue},
    VOLUME = {6},
      YEAR = {2023},
     PAGES = {907--940},
      ISSN = {2644-9463},
   MRCLASS = {60K35 (52C20 82C20)},
  MRNUMBER = {4648100},
MRREVIEWER = {Simone\ Baldassarri},
}

@misc{AggarwalToninelli26,
      title={Mixing times for Glauber dynamics of lozenge tilings of the hexagon}, 
      author={Amol Aggarwal and Fabio Toninelli},
      year={2026},
      eprint={2605.26664},
      archivePrefix={arXiv},
      primaryClass={math.PR},
}

@misc{Petrov_simulations,
  author       = {Leonid Petrov},
  title        = {Simulations},
  howpublished = {\url{https://lpetrov.cc/simulations/2025-12-08-triangular-dimers/}},
}

\end{document}